\documentclass[12pt]{amsart}

\usepackage{ucs}

\usepackage{amssymb}
\usepackage{amsthm}
\usepackage{amsmath}
\usepackage{latexsym}
\usepackage[cp1251]{inputenc}
\usepackage[mathcal]{eucal}
\usepackage{graphicx}
\usepackage{wrapfig}
\usepackage{caption}
\usepackage{subcaption}
\usepackage{indentfirst}
\usepackage[left=2.6cm,right=2.6cm,top=3cm,bottom=3cm,bindingoffset=0cm]{geometry}
\usepackage{enumerate}

\DeclareMathOperator{\aut}{Aut}

\DeclareMathOperator{\cay}{Cay}
\DeclareMathOperator{\cyc}{Cyc}

\DeclareMathOperator{\iso}{Iso}

\DeclareMathOperator{\orb}{Orb}

\DeclareMathOperator{\rk}{rk}

\DeclareMathOperator{\Span}{Span}

\DeclareMathOperator{\sym}{Sym}
\DeclareMathOperator{\rad}{rad}

\DeclareMathOperator{\alg}{Alg}

\makeatletter 
\def\@seccntformat#1{\csname the#1\endcsname. } 
\def\@biblabel#1{#1.}

\makeatother

\title{On separable abelian $p$-groups}

\author{Grigory Ryabov}
\address{Sobolev Institute of Mathematics, Novosibirsk, Russia}
\address{Novosibirsk State University, Novosibirsk, Russia}
\email{gric2ryabov@gmail.com}
\thanks{The work is supported by the Russian Foundation for Basic Research (project 17-51-53007)}

\date{}

\newtheorem{prop}{Proposition}[section]
\newtheorem{theo}{Theorem}[section]

\newtheorem{lemm}[prop]{Lemma}

\theoremstyle{definition}

\newtheorem*{rem}{Remark}

\begin{document}

\vspace{\baselineskip}
\vspace{\baselineskip}

\vspace{\baselineskip}

\vspace{\baselineskip}

\begin{abstract}
An $S$-ring (a Schur ring) is said to be \emph{separable} with respect to a class of groups $\mathcal{K}$ if every algebraic isomorphism from the $S$-ring in question to an $S$-ring over a group from $\mathcal{K}$ is induced by a combinatorial isomorphism. A finite group is said to be  \emph{separable} with respect to $\mathcal{K}$ if every $S$-ring over this group is separable with respect to $\mathcal{K}$. We provide a complete classification of  abelian $p$-groups separable with respect to the class of abelian groups.
\\
\\
\textbf{Keywords}: Isomorphisms, Schur rings, $p$-groups.
\\
\textbf{MSC}:05E30, 05C60, 20B35.
\end{abstract}

\maketitle

\section{Introduction}
Let $G$ be a finite group. A subring of the group ring $\mathbb{Z}G$ is called an \emph{$S$-ring} (a \emph{Schur ring}) over $G$ if it is determined in a natural way by a special partition of $G$ (the exact definition is given in Section~2). The classes of the partition  are called the \emph{basic sets} of the $S$-ring.  The concept of the $S$-ring goes back to  Schur and Wielandt. They used $S$-rings to study a permutation group containing a regular subgroup~\cite{Schur, Wi}. For more details on $S$-rings and their applications we refer the reader to~\cite{MP}.

Let $\mathcal{A}$ and $\mathcal{A}^{'}$ be $S$-rings over groups $G$  and  $G^{'}$ respectively. An \emph{algebraic isomorphism} from $\mathcal{A}$ to $\mathcal{A}^{'}$ is a ring isomorphism inducing a bijection between the basic sets of $\mathcal{A}$ and the basic sets of $\mathcal{A}^{'}$. Another type of an isomorphism of $S$-rings comes from  graph theory. A \emph{combinatorial isomorphism} from $\mathcal{A}$ to $\mathcal{A}^{'}$ is defined to be an isomorphism of the corresponding Cayley schemes (see Subsection~2.2). Every combinatorial isomorphism induces the algebraic one. However, the converse statement is not true (the corresponding examples can be found in~\cite{EP1}).

Let $\mathcal{K}$ be a class of groups. Following~\cite{EP3}, we say that an $S$-ring $\mathcal{A}$ is \emph{separable} with respect to $\mathcal{K}$ if every algebraic isomorphism from $\mathcal{A}$ to an $S$-ring over a group from $\mathcal{K}$ is induced by a combinatorial one. We call a finite group \emph{separable} with respect to $\mathcal{K}$ if every $S$-ring over $G$ is separable with respect to $\mathcal{K}$ (see~\cite{Ry2}).

The importance of separable $S$-rings comes from the following observation. Suppose that an $S$-ring $\mathcal{A}$ is separable with respect to $\mathcal{K}$. Then $\mathcal{A}$ is determined up to isomorphism in the class of $S$-rings over groups from $\mathcal{K}$ only by the tensor of its structure constants (with respect to the basis of $\mathcal{A}$ corresponding to the partition of the underlying group). 

Given a group $G$ denote the class of groups isomorphic to $G$ by $\mathcal{K}_G$. If $G$ is separable with respect to $\mathcal{K}_G$ then the isomorphism of two Cayley graphs over $G$ can be verified efficiently by  using the Weisfeiler-Leman algorithm~\cite{WeisL}. In the sense of~\cite{KPS} this means that the Weisfeiler-Leman dimension of the class of Cayley graphs over~$G$ is at most~3. More information concerned with separability and the graph isomorphism problem is presented in~\cite{EP3, Ry1}.

Denote the classes of cyclic and abelian groups by $\mathcal{K}_C$ and $\mathcal{K}_A$ respectively. The cyclic group of order~$n$ is denoted by $C_n$. In the present paper we are interested in abelian groups and especially in abelian $p$-groups which are separable with respect to $\mathcal{K}_A$. The problem of determining of all groups separable with respect to a given class $\mathcal{K}$  seems quite complicated even for $\mathcal{K}=\mathcal{K}_C$. Examples of cyclic groups which are non-separable with respect  to $\mathcal{K}_C$ were found in~\cite{EP1}. In~\cite{EP6} it was proved that cyclic $p$-groups are separable with respect to $\mathcal{K}_C$.  We prove that a similar statement is also true for $\mathcal{K}_A$.

\begin{theo}\label{main1}
For every prime $p$ a cyclic $p$-group is separable with respect to $\mathcal{K}_A$.
\end{theo}

The result obtained in~\cite{Ry2} implies that an abelian  group of order $4p$ is separable with respect to $\mathcal{K}_A$ for every prime~$p$. From~\cite{Klin} it follows that for every group $G$ of order at least~4 the group $G\times G$ is non-separable with respect to $\mathcal{K}_{G\times G}$.  One can check that a normal subgroup of a group separable with respect to $\mathcal{K}_A$ is separable with respect to $\mathcal{K}_A$ (see also Lemma~\ref{nonsepwr}). The above discussion shows that a non-cyclic abelian $p$-group separable with respect to $\mathcal{K}_A$ is isomorphic to  $C_p\times C_{p^k}$ or $C_p\times C_p\times C_{p^k}$, where $p\in \{2,3\}$ and $k\geq 1$. The separability of the groups from the first family was proved in~\cite{Ry1}. In the present paper we study the question on the separability of the groups from the second family.

\begin{theo}\label{main2}
The group  $C_p\times C_p\times C_{p^k}$, where $p\in\{2,3\}$ and $k\geq 1$, is separable with respect to $\mathcal{K}_A$ if and only if $k=1$.
\end{theo}

As an immediate consequence of Theorem~\ref{main1}, Theorem~\ref{main2}, and the above mentioned results, we obtain a complete classification of abelian $p$-groups separable with respect to $\mathcal{K}_A$.

\begin{theo}\label{main3}
An abelian $p$-group is separable with respect to $\mathcal{K}_A$  if and only if it is cyclic or isomorphic to one of the following groups:
$$C_2\times C_{2^k},~C_3\times C_{3^k},~C_2^3,~C_3^3,$$
where $k\geq 1$.
\end{theo}

Throughout the paper we write for short ``separable''   instead of ``separable with respect to $\mathcal{K}_A$''. The text is organized in the following way. Section~2  contains a background of $S$-rings. Section~3 is devoted to $S$-rings over cyclic $p$-groups. We finish Section~3 with the proof of Theorem~\ref{main1}. In Section~4 we prove Theorem~\ref{main2}.

The author would like to thank Prof. I. Ponomarenko for the fruitful discussions on the subject matters  and  Dr. Sven Reichard for the help with computer calculations.

$$$$

{\bf Notation.}

The ring of rational integers is denoted by $\mathbb{Z}$.

Let $X\subseteq G$. The element $\sum_{x\in X} {x}$ of the group ring $\mathbb{Z}G$ is denoted by $\underline{X}$.

The order of $g\in G$ is denoted by $|g|$.

The set $\{x^{-1}:x\in X\}$ is denoted by $X^{-1}$.

The subgroup of $G$ generated by $X$ is denoted by $\langle X\rangle$; we also set $\rad(X)=\{g\in G:\ gX=Xg=X\}$.

If  $m\in \mathbb{Z}$ then the set $\{x^m: x \in X\}$ is denoted by $X^{(m)}$.

Given a set $X\subseteq G$ the set $\{(g,xg): x\in  X, g\in G\}$ of edges of the Cayley graph $\cay(G,X)$ is denoted by $R(X)$.

The group of all permutations of a set $\Omega$ is denoted by $\sym(\Omega)$.

The subgroup of $\sym(G)$ induced by right multiplications of $G$ is denoted by $G_{right}$.

For a set $\Delta\subseteq \sym(G)$ and a section $S=U/L$ of $G$ we set 
$$\Delta^S=\{f^S:~f\in \Delta,~S^f=S\},$$
where $S^f=S$ means that $f$ permutes the $L$-cosets in $U$ and $f^S$ denotes the bijection of $S$ induced by $f$.

If a group $K$ acts on a set $\Omega$ then the set of all orbtis of $K$ on $\Omega$ is denoted by $\orb(K,\Omega)$.

If $H\leq G$ then the normalizer of $H$ in $G$ is denoted by $N_G(H)$.

If $K\leq \sym(\Omega)$ and $\alpha\in \Omega$ then the stabilizer of $\alpha$ in $K$ is denoted by $K_{\alpha}$.

The cyclic group of order $n$ is denoted by  $C_n$.

\section{$S$-rings}
In this section we give a background of $S$-rings. The most of definitions and statements presented here are taken from~\cite{MP,Ry1}.

\subsection{Definitions and basic facts}

Let $G$ be a finite group and $\mathbb{Z}G$  the group ring over the integers. The identity element of $G$ is denoted by $e$.  A subring  $\mathcal{A}\subseteq \mathbb{Z} G$ is called an \emph{$S$-ring} over $G$ if there exists a partition $\mathcal{S}=\mathcal{S}(\mathcal{A})$ of~$G$ such that:

$(1)$ $\{e\}\in\mathcal{S}$,

$(2)$  if $X\in\mathcal{S}$ then $X^{-1}\in\mathcal{S}$,

$(3)$ $\mathcal{A}=\Span_{\mathbb{Z}}\{\underline{X}:\ X\in\mathcal{S}\}$.

The elements of $\mathcal{S}$ are called the \emph{basic sets} of  $\mathcal{A}$ and the number $|\mathcal{S}|$ is called the \emph{rank} of~$\mathcal{A}$. Given $X,Y,Z\in\mathcal{S}$ the number of distinct representations of $z\in Z$ in the form $z=xy$ with $x\in X$ and $y\in Y$ is denoted by $c^Z_{X,Y}$. If $X$ and $Y$ are basic sets of $\mathcal{A}$ then $\underline{X}~\underline{Y}=\sum_{Z\in \mathcal{S}(\mathcal{A})}c^Z_{X,Y}\underline{Z}$. So the integers  $c^Z_{X,Y}$ are structure constants of $\mathcal{A}$ with respect to the basis $\{\underline{X}:\ X\in\mathcal{S}\}$. It is easy to verify that given basic  sets $X$ and $Y$ the set $XY$ is also basic whenever  $|X|=1$ or $|Y|=1$.

A set $X \subseteq G$ is said to be an \emph{$\mathcal{A}$-set} if $\underline{X}\in \mathcal{A}$. A subgroup $H \leq G$ is said to be an \emph{$\mathcal{A}$-subgroup} if $H$ is an $\mathcal{A}$-set. One can check that for every $\mathcal{A}$-set $X$ the groups $\langle X \rangle$ and $\rad(X)$ are $\mathcal{A}$-subgroups.

A section $U/L$ is said to be an \emph{$\mathcal{A}$-section} if $U$ and $L$ are $\mathcal{A}$-subgroups. If $S=U/L$ is an $\mathcal{A}$-section then the module
$$\mathcal{A}_S=Span_{\mathbb{Z}}\left\{\underline{X}^{\pi}:~X\in\mathcal{S}(\mathcal{A}),~X\subseteq U\right\},$$
where $\pi:U\rightarrow U/L$ is the canonical epimorphism, is an $S$-ring over $S$.

 If $K \leq \aut(G)$ then the set  $\orb(K,G)$ forms a partition of  $G$ that defines an  $S$-ring $\mathcal{A}$ over $G$.  In this case  $\mathcal{A}$ is called \emph{cyclotomic} and denoted by $\cyc(K,G)$.

Let $G$ be abelian. Then from Schur's result~\cite{Schur} it follows that $X^{(m)}\in \mathcal{S}(\mathcal{A})$ for every  $X\in \mathcal{S}(\mathcal{A})$ and every $m$ coprime to $|G|$. We say that $X,Y\in \mathcal{S}(\mathcal{A})$ are \emph{rationally conjugate} if $Y=X^{(m)}$ for some $m$ coprime to $|G|$.

\subsection{Isomorphisms and schurity}

Throughout this and the next two subsections $\mathcal{A}$ and $\mathcal{A}^{'}$ are $S$-rings over groups $G$ and $G^{'}$ respectively. A bijection $f:G\rightarrow G^{'}$ is called \emph{a (combinatorial) isomorphism } from $\mathcal{A}$ over to $\mathcal{A}^{'}$  if 
$$\{R(X)^f: X\in \mathcal{S}(\mathcal{A})\}=\{R(X^{'}): X^{'}\in \mathcal{S}(\mathcal{A}^{'})\},$$ where $R(X)^f=\{(g^f,~h^f):~(g,~h)\in R(X)\}$. If there exists an  isomorphism from  $\mathcal{A}$  to $\mathcal{A}^{'}$ we write $\mathcal{A}\cong\mathcal{A}^{'}$. The group $\iso(\mathcal{A})$ of all isomorphisms from $\mathcal{A}$ onto itself has a normal subgroup
$$\aut(\mathcal{A})=\{f\in \iso(\mathcal{A}): R(X)^f=R(X)~\text{for every}~X\in \mathcal{S}(\mathcal{A})\}.$$
This subgroup is called the \emph{automorphism group} of $\mathcal{A}$. Note that $\aut(\mathcal{A})\geq G_{right}$. If $S$ is an $\mathcal{A}$-section then $\aut(\mathcal{A})^S\leq\aut(\mathcal{A}_S)$. An $S$-ring $\mathcal{A}$ over $G$ is said to be \emph{normal} if $G_{right}\trianglelefteq \aut(\mathcal{A})$. One can check that 
$$N_{\aut(\mathcal{A})}(G_{right})_e=\aut(\mathcal{A})\cap \aut(G).~\eqno(1)$$

Now let $K$ be a subgroup of $\sym(G)$ containing  $G_{right}$. As Schur proved in~\cite{Schur}, the $\mathbb{Z}$-submodule
$$V(K,G)=\Span_{\mathbb{Z}}\{\underline{X}:~X\in \orb(K_e,~G)\},$$
is an $S$-ring over $G$. An $S$-ring $\mathcal{A}$ over  $G$ is called \emph{schurian} if $\mathcal{A}=V(K,G)$ for some $K$ such that $G_{right}\leq K \leq \sym(G)$. Not every $S$-ring is schurian. The first example of a non-schurian $S$-ring was found by Wielandt in~\cite[Theorem~25.7]{Wi}. It is easy to see that $\mathcal{A}$ is schurian if and only if 
$$\mathcal{S}(\mathcal{A})=\orb(\aut(\mathcal{A})_e,G).~\eqno(2)$$
Every cyclotomic $S$-ring is schurian. More precisely, if $\mathcal{A}=\cyc(K,G)$ for some $K\leq \aut(G)$ then $\mathcal{A}=V(G_{right}\rtimes K,G)$.

\subsection{ Algebraic isomorphisms and separability}

 A bijection $\varphi:\mathcal{S}(\mathcal{A})\rightarrow\mathcal{S}(\mathcal{A}^{'})$ is called an \emph{algebraic isomorphism} from $\mathcal{A}$  to $\mathcal{A}^{'}$ if 
$$c_{X,Y}^Z=c_{X^{\varphi},Y^{\varphi}}^{Z^{\varphi}}$$ 
for all $X,Y,Z\in \mathcal{S}(\mathcal{A})$. The mapping $\underline{X}\rightarrow \underline{X}^{\varphi}$ is extended by linearity to the ring isomorphism of $\mathcal{A}$  and $\mathcal{A}^{'}$. This ring isomorphism we denote also by $\varphi$. If there exists an algebraic isomorphism from  $\mathcal{A}$  to $\mathcal{A}^{'}$ then we write $\mathcal{A}\cong_{\alg}\mathcal{A}^{'}$. An algebraic isomorphism from $\mathcal{A}$ to itself is called an \emph{algebraic automorphism} of $\mathcal{A}$. The group of all algebraic automorphisms of $\mathcal{A}$ is denoted by $\aut_{\alg}(\mathcal{A})$. 

Every isomorphism $f$ of $S$-rings  preserves  the structure constants  and hence $f$ induces the algebraic isomorphism $\varphi_f$. However, not every algebraic isomorphism is induced by a combinatorial one (see~\cite{EP1}). Let $\mathcal{K}$ be a class of groups. An $S$-ring $\mathcal{A}$ is defined to be \emph{separable} with respect to $\mathcal{K}$ if every  algebraic isomorphism from $\mathcal{A}$ to an $S$-ring over a group from $\mathcal{K}$ is induced by a combinatorial isomorphism. 

Put  
$$\aut_{\alg}(\mathcal{A})_0=\{\varphi\in \aut_{\alg}(\mathcal{A}): \varphi=\varphi_f~\text{for some}~f\in \iso(\mathcal{A})\}.$$
It is easy to see that $\varphi_f=\varphi_g$ for $f,g\in \iso(\mathcal{A})$ if and only if $gf^{-1}\in \aut(\mathcal{A})$. Therefore
$$|\aut_{\alg}(\mathcal{A})_0|=|\iso(\mathcal{A})|/|\aut(\mathcal{A})|.~\eqno(3)$$

One can verify that for every group $G$ the $S$-ring of rank~2 over $G$ and $\mathbb{Z}G$ are separable with respect to the class of all finite groups. In the former case  there exists the unique algebraic isomorphism from the $S$-ring of rank~2 over $G$ to the $S$-ring of rank~2 over a given  group of order $|G|$ and this algebraic isomorphism is induced by every bijection. In the latter case every basic set is singleton and hence every algebraic isomorphism is induced by an isomorphism in a natural way.

Let $\varphi:\mathcal{A}\rightarrow \mathcal{A}^{'}$ be an algebraic isomorphism. One can check that $\varphi$ is extended to a bijection between  $\mathcal{A}$- and $\mathcal{A}^{'}$-sets and hence between  $\mathcal{A}$- and $\mathcal{A}^{'}$-sections. The images of an $\mathcal{A}$-set $X$ and an $\mathcal{A}$-section $S$ under these extensions are denoted by $X^{\varphi}$ and $S^{\varphi}$ respectively. If $S$ is an $\mathcal{A}$-section then  $\varphi$ induces the algebraic isomorphism $\varphi_S:\mathcal{A}_S\rightarrow \mathcal{A}^{'}_{S^{'}}$, where $S^{'}=S^{\varphi}$. The above bijection between the $\mathcal{A}$- and $\mathcal{A}^{'}$-sets is, in fact, an isomorphism of the corresponding lattices. One can check that 
$$ \langle X^{\varphi} \rangle = \langle X \rangle ^{\varphi}~\text{and}~\rad(X^{\varphi})=\rad(X)^{\varphi}~$$for every  $\mathcal{A}$-set $X$ (see~\cite[(10)]{EP5}). Since $c^{\{e\}}_{X,Y}=\delta_{Y,X^{-1}}|X|$, where $X,Y\in \mathcal{S}(\mathcal{A})$ and $\delta_{Y,X^{-1}}$ is the Kronecker delta,  we conclude that $|X|=c^{\{e\}}_{X,X^{-1}}$, $(X^{-1})^{\varphi}=(X^{\varphi})^{-1}$, and $|X|=|X^{\varphi}|$ for every $\mathcal{A}$-set $X$. In particular, $|G|=|G^{'}|$.

\subsection{Cayley isomorphisms}

A group isomorphism $f:G\rightarrow G^{'}$ is called a \emph{Cayley isomorphism} from  $\mathcal{A}$  to $\mathcal{A}^{'}$   if $\mathcal{S}(\mathcal{A})^f=\mathcal{S}(\mathcal{A}^{'})$. If there exists a Cayley isomorphism from  $\mathcal{A}$  to $\mathcal{A}^{'}$ we write $\mathcal{A}\cong_{\cay}\mathcal{A}^{'}$. Every Cayley isomorphism is a (combinatorial) isomorphism, however the converse statement is not true.

\subsection{Algebraic fusions}

Let $\mathcal{A}$ be an $S$-ring over $G$ and $\Phi\leq \aut_{\alg}(\mathcal{A})$. Given $X\in \mathcal{S}(\mathcal{A})$ put $X^{\Phi}=\bigcup \limits_{\varphi\in \Phi} X^{\varphi}$. The partition
$$\{X^{\Phi}: X\in \mathcal{S}(\mathcal{A})\}$$
defines an $S$-ring over $G$ called the \emph{algebraic fusion} of $\mathcal{A}$ with respect to $\Phi$ and denoted by $\mathcal{A}^{\Phi}$. Suppose that $\Phi=\{\varphi_f: f\in K\}$ for some $K\leq \iso(\mathcal{A})$ and $\mathcal{A}$ is schurian. Then one can verify that 
$$\mathcal{A}^{\Phi}=V(\aut(\mathcal{A})K, G).$$
In particular, the following statement holds.

\begin{lemm}\label{fusion}
Let $\mathcal{A}$ be a schurian $S$-ring over $G$ and $K\leq \iso(\mathcal{A})$. Then $\mathcal{A}^{\Phi}$, where $\Phi=\{\varphi_f: f\in K\}$, is also schurian.
\end{lemm}

\subsection{Wreath product and tensor products}

Let $\mathcal{A}$ be an $S$-ring over a group $G$ and $S=U/L$  an $\mathcal{A}$-section. The $S$-ring~$\mathcal{A}$ is called the \emph{$S$-wreath product}  if $L\trianglelefteq G$ and $L\leq\rad(X)$ for all basic sets $X$ outside~$U$. In this case we write 
$$\mathcal{A}=\mathcal{A}_U \wr_S \mathcal{A}_{G/L}.$$
The $S$-wreath product is called \emph{non-trivial} or \emph{proper}  if $e\neq L$ and $U\neq G$. If $U=L$ we say that $\mathcal{A}$ is the \emph{wreath product} of  $\mathcal{A}_L$ and $\mathcal{A}_{G/L}$ and write $\mathcal{A}=\mathcal{A}_L\wr\mathcal{A}_{G/L}$.

Let $\mathcal{A}_1$ and $\mathcal{A}_2$ be $S$-rings over groups $G_1$ and $G_2$ respectively. Then the set
$$\mathcal{S}=\mathcal{S}(\mathcal{A}_1)\times \mathcal{S}(\mathcal{A}_2)=\{X_1\times X_2:~X_1\in \mathcal{S}(\mathcal{A}_1),~X_2\in \mathcal{S}(\mathcal{A}_2)\} $$
forms a partition of  $G=G_1\times G_2$ that defines an  $S$-ring over $G$. This $S$-ring is called the  \emph{tensor product}  of $\mathcal{A}_1$ and $\mathcal{A}_2$ and denoted by $\mathcal{A}_1 \otimes \mathcal{A}_2$.

\begin{lemm} \label{schurtens}
The	 tensor product of two separable $S$-rings is separable.
\end{lemm}
	
\begin{proof}
As noted in~\cite[Lemma~2.6]{Ry2}, the statement of the lemma follows from~\cite[Theorem~1.20]{E}.
\end{proof}

\begin{lemm}\cite[Lemma 4.4]{Ry1}\label{sepwr}
Let $\mathcal{A}$ be the $S$-wreath product over an abelian group $G$ for some $\mathcal{A}$-section $S=U/L$. Suppose that $\mathcal{A}_U$ and $\mathcal{A}_{G/L}$ are separable and $\aut(\mathcal{A}_U)^S=\aut(\mathcal{A}_S)$. Then $\mathcal{A}$ is separable. In particular, the wreath product of two separable $S$-rings is separable.
\end{lemm}

Let $\Omega$ be a finite set. Permutation groups $K,~K^{'}\leq \sym(\Omega)$   are called  2-\emph{equivalent} if  $\orb(K,\Omega^2)=\orb(K^{'},\Omega^2)$. A permutation group  $K\leq \sym(\Omega)$ is called 2-\emph{isolated} if  it is the only group which is 2-equivalent to $K$.

\begin{lemm}\label{2isol1}
 Let $\mathcal{A}$ be the $S$-wreath product over an abelian group $G$ for some $\mathcal{A}$-section $S=U/L$. Suppose that $\mathcal{A}_U$ and $\mathcal{A}_{G/L}$ are separable, $\mathcal{A}_U$ is schurian, and the group $\aut(\mathcal{A}_S)$ is 2-isolated. Then $\mathcal{A}$ is separable.
\end{lemm}

\begin{proof}
Since $\mathcal{A}_U$ is schurian, the groups $\aut(\mathcal{A}_U)^S$ and $\aut(\mathcal{A}_S)$ are 2-equivalent. Indeed, 
$$\orb(\aut(\mathcal{A}_U)^S, S^2)=\orb(\aut(\mathcal{A}_S),S^2)=\{R(X): X\in \mathcal{S}(\mathcal{A}_S)\}.$$
This implies that $\aut(\mathcal{A}_U)^S=\aut(\mathcal{A}_S)$ because $\aut(\mathcal{A}_S)$ is 2-isolated. Therefore the conditions of Lemma~\ref{sepwr} hold and $\mathcal{A}$ is separable.
\end{proof}

\begin{lemm}\label{nonsepwr}
Let $H$ be a normal subgroup of a group $G$, $\mathcal{B}$  an $S$-ring over $H$, $\varphi\in \aut_{\alg}(\mathcal{B})\setminus \aut_{\alg}(\mathcal{B})_0$. Then there exists  $\psi\in \aut_{\alg}(\mathcal{A})\setminus \aut_{\alg}(\mathcal{A})_0$, where $\mathcal{A}=\mathcal{B}\wr \mathbb{Z}(G/H)$, such that $\psi^H=\varphi$.
\end{lemm}

\begin{proof}
Define $\psi$ as follows: $X^{\psi}=X^{\varphi}$ for $X\in \mathcal{S}(\mathcal{A}_H)$ and $X^{\psi}=X$ for $X\in \mathcal{S}(\mathcal{A})\setminus\mathcal{S}(\mathcal{A}_H)$. Let us prove that $\psi\in \aut_{\alg}(\mathcal{A})$. To do this it suffices to check that $c_{X^{\psi},Y^{\psi}}^{Z^{\psi}}=c_{X,Y}^Z$ for all $X,Y,Z\in \mathcal{S}(\mathcal{A})$. Suppose that $X,Y\in \mathcal{S}(\mathcal{A}_H)$. If $Z\in \mathcal{S}(\mathcal{A}_H)$ then $c_{X^{\psi},Y^{\psi}}^{Z^{\psi}}=c_{X^{\varphi},Y^{\varphi}}^{Z^{\varphi}}=c_{X,Y}^Z$. If $Z\notin \mathcal{S}(\mathcal{A}_H)$ then $Z^{\psi}\notin \mathcal{S}(\mathcal{A}_H)$ and hence $c_{X^{\psi},Y^{\psi}}^{Z^{\psi}}=c_{X,Y}^Z=0$. 

Now suppose that exactly one of the sets $X,Y$, say $X$, lies inside $H$. Then $Y^{\psi}=Y$ and $X\cup X^{\psi}\subseteq H\leq \rad(Y)$. So $\underline{X}\underline{Y}=\underline{X}^{\psi}\underline{Y}=|X|\underline{Y}$. This implies that $c_{X^{\psi},Y^{\psi}}^{Z^{\psi}}=c_{X,Y}^Z=|X|$ whenever $Z=Y$ and $c_{X^{\psi},Y^{\psi}}^{Z^{\psi}}=c_{X,Y}^Z=0$ otherwise.

Finally, suppose that $X,Y \notin \mathcal{S}(\mathcal{A}_H)$. In this case $X^{\psi}=X$ and $Y^{\psi}=Y$. If $Z\notin \mathcal{S}(\mathcal{A}_H)$ then $Z^{\psi}=Z$ and hence $c_{X^{\psi},Y^{\psi}}^{Z^{\psi}}=c_{X,Y}^Z$. If $Z\in \mathcal{S}(\mathcal{A}_H)$ then $Z$ and $Z^{\psi}$ enter the element $\underline{X}\underline{Y}$ with the same coefficients because $H=\rad(X)\cap \rad(Y)$. Therefore $c_{X^{\psi},Y^{\psi}}^{Z^{\psi}}=c_{X,Y}^Z$. Thus, $\psi\in \aut_{\alg}(\mathcal{A})$.

If $\psi$ is induced by an isomorphism  then \cite[Lemma~3.4]{EP5} implies that $\psi^H=\varphi$ is also induced by an isomorphism. We obtain a contradiction with the assumption of the lemma and the lemma is proved. 
\end{proof}

\section{$S$-rings over cyclic $p$-groups}

In this section we prove Theorem~\ref{main1}. Before the proof we recall some results on $S$-rings over cyclic $p$-groups. The most of them can be found in~\cite{EP2, EP4}. Throughout the section $p$ is an odd prime, $G$ is a cyclic $p$-group and $\mathcal{A}$ is an $S$-ring over $G$.  We say that $X\in \mathcal{A}$ is \emph{highest} if $X$ contains a generator of $G$. Put   $\rad(\mathcal{A})=\rad(X)$, where $X$ is highest. Note that $\rad(\mathcal{A})$ does not depend on the choice of $X$ because every two basic sets are rationally conjugate and hence have the same radicals.

\begin{lemm}\label{p1}
The $S$-ring $\mathcal{A}$ is schurian and one of the following statements holds for $\mathcal{A}$:

$(1)$ $|\rad(\mathcal{A})|=1$ and $\rk(\mathcal{A})=2$;

$(2)$ $|\rad(\mathcal{A})|=1$, $\mathcal{A}$ is normal, and $\mathcal{A}=\cyc(K,G)$ for some $K\leq K_0$, where $K_0$ is the subgroup of $\aut(G)$ of order $p-1$;

$(3)$ $|\rad(\mathcal{A})|>1$ and $\mathcal{A}$ is the proper generalized wreath product.

\end{lemm}

\begin{proof}
The $S$-ring $\mathcal{A}$ is schurian by the main result of~\cite{Po}. The other statements of the lemma follow from~\cite[Theorem~4.1, Theorem~4.2 (1)]{EP4} and~\cite[Lemma~5.1, (1)]{EP2}.
\end{proof}

\begin{lemm}\label{section}
Let $S$ be an $\mathcal{A}$-section with $|S|\geq p^2$. The following statements hold:

(1) If Statement~2 of Lemma~\ref{p1} holds for $\mathcal{A}$ then Statement~2 of Lemma~\ref{p1} holds for $\mathcal{A}_S$; 

(2) If $\rk(\mathcal{A}_S)=2$ then $\aut(\mathcal{A})^S=\sym(S)$.
\end{lemm}

\begin{proof}
Statement~1 of the lemma follows from~\cite[Corollary~4.4]{EP4}  and Statement~2 of the lemma follows from~\cite[Theorem~4.6,(1)]{EP4}.
\end{proof}

\begin{lemm}\label{2isol}
Suppose that Statement~2 of Lemma~\ref{p1} holds for $\mathcal{A}$. Then $\aut(\mathcal{A})$ is 2-isolated.
\end{lemm}

\begin{proof}
By~\cite[Lemma~8.2]{MP3}, it suffices to prove that $\aut(\mathcal{A})_e$ has a faithful regular orbit. The $S$-ring $\mathcal{A}$ is normal. So~(1) implies that $\aut(\mathcal{A})_e\leq \aut(G)$. Let $X\in\mathcal{S}(\mathcal{A})$ be highest. Since $\mathcal{A}$ is cyclotomic, each element of $X$ is a generator of $G$. If $f\in \aut(\mathcal{A})_e$ fixes some $x\in X$ then $f$ is trivial because $f\in \aut(G)$ and $x$ is a generator of $G$. Besides, $\mathcal{A}$ is schurian and hence $X\in \orb(\aut(\mathcal{A})_e,G)$ by~(2). Therefore $X$ is a regular orbit of $\aut(\mathcal{A})_e$. The group $\aut(G)$ is cyclic because $p$ is odd. So both of the groups $\aut(\mathcal{A})_e$ and $\aut(\mathcal{A})_e^X$ are cyclic groups of order $|X|$. Thus, $X$ is a faithful regular orbit of $\aut(\mathcal{A})_e$ and the lemma is proved.
\end{proof}

\begin{lemm}\label{trivradsep}
Suppose that Statement~2 of Lemma~\ref{p1} holds for $\mathcal{A}$ and $\varphi$ is an algebraic isomorphism from $\mathcal{A}$ to an $S$-ring $\mathcal{A}^{'}$ over an abelian group $G^{'}$. Then $G^{'}$ is cyclic.
\end{lemm}

\begin{proof}
By the hypothesis,
$$\mathcal{A}=\cyc(K,G)~\text{for some}~K\leq \aut(G)~\text{with}~|K|\leq p-1.$$ 
The group $E=\{g\in G: |g|=p\}$ is an $\mathcal{A}$-subgroup of order $p$ because  $\mathcal{A}$ is cyclotomic. The group $E^{'}=E^{\varphi}$ is an $\mathcal{A}^{'}$-subgroup of order $p$ by the properties of an algebraic isomorphism. Assume that $G^{'}$ is non-cyclic. Then there exists $X^{'}\in \mathcal{S}(\mathcal{A}^{'})$ containing an element of order $p$ outside $E^{'}$. Let $X\in\mathcal{S}(\mathcal{A})$ such that $X^{\varphi}=X^{'}$. The set $X$ consists of elements  of order greater than $p$ because $G$ is cyclic and all elements of order $p$ from $G$ lie inside $E$. The identity element $e$ of $G$ enters the element $\underline{X}^p$ with a coefficient dividing by~$p$ because $x^p\neq e$ for each $x\in X$. The identity element $e^{'}$ of $G^{'}$ enters the element $(\underline{X}^{'})^p$ with a coefficient which  is not divided by~$p$ because $(x^{'})^p=e^{'}$ for some $x^{'}\in X^{'}$ and $|X^{'}|\leq p-1$. Since $\varphi$ is an algebraic isomorphism, we have 
$$(\underline{X}^p)^{\varphi}=(\underline{X}^{'})^p~\text{and}~\{e\}^{\varphi}=\{e^{'}\}.$$  
This implies that $e$ and $e^{'}$ must enter $\underline{X}^p$ and $(\underline{X}^{'})^p$ respectively with the same coefficients, a contradiction. Therefore $G^{'}$ is cyclic and the lemma is proved.
\end{proof}

\begin{lemm}\label{wreathcyclic}
Suppose that $|\rad(\mathcal{A})|>1$. Then there exists an  $\mathcal{A}$-section  $S=U/L$ such that $\mathcal{A}$ is the proper $S$-wreath product, $|\rad(\mathcal{A}_U)|=1$, and $|L|=p$.
\end{lemm}

\begin{proof}
From \cite[Lemma~5.2]{Ry1} it follows that there exists an  $\mathcal{A}$-section  $U/L_1$ such that $\mathcal{A}$ is the proper $U/L_1$-wreath product and $|\rad(\mathcal{A}_U)|=1$. Let $L$ be a subgroup of $L_1$ of order $p$. Then the lemma holds for $S=U/L$.
\end{proof}

\begin{lemm}\cite[Theorem~1.3]{EP6}\label{cycsep}
Every $S$-ring over a cyclic $p$-group is separable with respect to $\mathcal{K}_C$.
\end{lemm}

\begin{proof}[Proof of the Theorem \ref{main1}]
The statement of the theorem for $p\in\{2,3\}$ was proved in~\cite[Lemma~5.5]{Ry1}. Further we assume that $p\geq 5$. Let $\mathcal{A}$ be an $S$-ring over a cyclic $p$-group $G$ of order $p^k$, where $k\geq 1$. Prove that $\mathcal{A}$ is separable. We proceed by induction on~$k$. If $k=1$ then $G$ is the unique up to isomorphism group of order~$p$ and the statement of the theorem follows from Lemma~\ref{cycsep}. 

Let  $k\geq 2$. One of the statements of Lemma~\ref{p1} holds for $\mathcal{A}$. If Statement~1 of Lemma~\ref{p1} holds for $\mathcal{A}$ then $\rk(\mathcal{A})=2$ and hence $\mathcal{A}$ is separable. Suppose that Statement~2 of Lemma~\ref{p1} holds for $\mathcal{A}$. Let $\varphi$ be an algebraic isomorphism from $\mathcal{A}$ to an $S$-ring $\mathcal{A}^{'}$ over an abelian group $G^{'}$. Due to Lemma~\ref{trivradsep}, the group $G^{'}$ is cyclic. So $\varphi$ is induced by an isomorphism by Lemma~\ref{cycsep}. Therefore  $\mathcal{A}$ is separable.

Now suppose that Statement~3 of Lemma~\ref{p1} holds for $\mathcal{A}$. Then $\mathcal{A}=\mathcal{A}_U \wr_S \mathcal{A}_{G/L}$ for some $\mathcal{A}$-section $S=U/L$ with $L>e$ and $U<G$. The $S$-rings $\mathcal{A}_U$ and $\mathcal{A}_{G/L}$ are separable by the induction hypothesis. Due to Lemma~\ref{wreathcyclic} we may assume that $\rad(\mathcal{A}_U)=e$ and $L=p$. In this case $\rk(\mathcal{A}_U)=2$ or Statement~2 of Lemma~\ref{p1} holds for $\mathcal{A}_U$. If $\rk(\mathcal{A}_U)=2$ or $|S|=1$ then $U=L$  and $\mathcal{A}$ is separable by Lemma~\ref{sepwr}. 

Assume that Statement~2 of Lemma~\ref{p1} holds for $\mathcal{A}_U$. If $|S|\geq p^2$ then Statement~2 of Lemma~\ref{p1} holds for $\mathcal{A}_S$ by Statement~1 of Lemma~\ref{section}. Lemma~\ref{2isol} implies that $\aut(\mathcal{A}_S)$ is 2-isolated. The $S$-ring $\mathcal{A}_U$  is cyclotomic and hence it is schurian. Therefore $\mathcal{A}$ is separable by Lemma~\ref{2isol1}.

It remains to consider the case when $|S|=p$. In this case $|U|=p^2$. If $\rad(X)>L$ for every $X\in \mathcal{S}(\mathcal{A})$ outside $U$ then $\rad(X)\geq U$ for every $X\in \mathcal{S}(\mathcal{A})$ outside $U$ because $G$ is cyclic.  This yields that $\mathcal{A}=\mathcal{A}_U \wr \mathcal{A}_{G/U}$ and hence  $\mathcal{A}$ is separable by Lemma~\ref{sepwr}.

Suppose that there exists $X\in \mathcal{S}(\mathcal{A})$ outside $U$ with $\rad(X)=L$. The remaining part of the proof is divided into two cases.

\textbf{Case 1:} $\langle X \rangle<G$. In this  case put $S_1=\langle X \rangle /L$. The $S$-ring $\mathcal{A}$ is the $S_1$-wreath product and $|S_1|\geq p^2$. Note that $|\rad(\mathcal{A}_{S_1})|=1$ because $\rad(X)=L$. So Statement~1 or Statement~2 of Lemma~\ref{p1} holds for $\mathcal{A}_{S_1}$. In the former case $\aut(\mathcal{A}_{\langle X \rangle})^{S_1}=\aut(\mathcal{A}_{S_1})=\sym(S_1)$ by Statement~2 of Lemma~\ref{section} and $\mathcal{A}$ is separable by Lemma~\ref{sepwr}. In the latter case $\aut(\mathcal{A}_{S_1})$ is 2-isolated by Lemma~\ref{2isol}. Since $\mathcal{A}_{\langle X \rangle}$ is schurian, the conditions of Lemma~\ref{2isol1} hold for $S_1$ and $\mathcal{A}$ is separable by Lemma~\ref{2isol1}. 

\textbf{Case 2:} $\langle X \rangle=G$. In this case $|\rad(\mathcal{A}_{G/L})|=1$ because $\rad(X)=L$. Let $\pi: G\rightarrow G/L$ be the canonical epimorphism. Clearly, $\pi(U)$ is an $\mathcal{A}_{G/L}$-subgroup and $\pi(X)$ lies outside $\pi(U)$. So $\rk(\mathcal{A}_{G/L})>2$ and hence Statement~2 of Lemma~\ref{p1} holds for $\mathcal{A}_{G/L}$.  

Let $\varphi$ be an algebraic isomorphism from $\mathcal{A}$ to an $S$-ring $\mathcal{A}^{'}$ over an abelian group $G^{'}$. Put $U^{'}=U^{\varphi}$ and $L^{'}=L^{\varphi}$. Clearly,
$$L^{'}\leq U^{'}.~\eqno(4)$$
The algebraic isomorphism $\varphi$ induces the algebraic isomorphism $\varphi_U$ from $\mathcal{A}_{U}$ to  $\mathcal{A}_{U^{'}}$, where $U^{'}=U^{\varphi}$. From Lemma~\ref{trivradsep} it follows that 
$$U^{'}\cong C_{p^2}.~\eqno(5)$$ 
Also $\varphi$ induces the algebraic isomorphism $\varphi_{G/L}$ from $\mathcal{A}_{G/L}$ to  $\mathcal{A}_{G^{'}/L^{'}}$. Lemma~\ref{trivradsep} implies that $G^{'}/L^{'}$ is cyclic. Since $|L^{'}|=|L|=p$, we conclude that
$$G^{'}\cong C_{p^k}~\text{or}~G^{'}\cong C_p \times C_{p^{k-1}}.$$
However, in the latter case $L^{'}$ is not contained in a cyclic group of order $p^2$ because $G^{'}/L^{'}$ is cyclic. This contradicts to~(4) and~(5). So $G^{'}$ is cyclic and $\varphi$ is induced by an isomorphism by Lemma~\ref{cycsep}. Therefore $\mathcal{A}$ is separable and the theorem is proved.
\end{proof}

\section{Proof of Theorem~\ref{main2}}

\begin{prop}\label{k1}
The group $C_p^3$ is separable for $p\in\{2,3\}$.
\end{prop}

Before we prove Propostion~\ref{k1} we give the  lemma providing a description of $S$-rings over these groups.

\begin{lemm}\label{descr}
Let $\mathcal{A}$ be an $S$-ring over $C_p^3$, where $p\in\{2,3\}$. Then $\mathcal{A}$ is schurian and one of the following statements holds:

$(1)$ $\rk(\mathcal{A})=2$;

$(2)$ $\mathcal{A}$ is the tensor product of smaller $S$-rings;

$(3)$ $\mathcal{A}$ is the proper $S$-wreath product of two $S$-rings with $|S|\leq p$;

$(4)$ $p=3$ and $\mathcal{A}\cong_{\cay} \mathcal{A}_i$, where $\mathcal{A}_i$ is one of the~14 exceptional $S$-rings whose parameters are listed in Table~1.
\end{lemm}

\begin{rem}
In Table~1 the notation $k^m$ means that an $S$-ring have exactly $m$ basic sets of size $k$.
\end{rem}

\begin{table}

{\small
\begin{tabular}{|l|l|l|}
  \hline
  $S$-ring & rank & sizes of basic sets    \\
  \hline
  $\mathcal{A}_1$ & 3 & $1$, $13^2$\\ \hline
  $\mathcal{A}_2$ & 4 & $1$, $6$, $8$, $12$ \\  \hline
  $\mathcal{A}_3$ & 4 &  $1$, $2$, $12^2$\\  \hline
	$\mathcal{A}_4$ & 5 & $1$, $4^2$, $6$, $12$\\  \hline
	$\mathcal{A}_5$ & 5 & $1$, $2$, $8^3$\\  \hline
	$\mathcal{A}_6$ & 6 & $1$, $2$, $6^4$\\ \hline
	$\mathcal{A}_7$ & 7 & $1$, $2$, $4^4$, $8$\\ \hline
	$\mathcal{A}_8$ & 7 & $1$, $2$, $3^2$, $6^3$\\ \hline
	$\mathcal{A}_9$ & 8 & $1$, $2$, $4^6$\\ \hline
	$\mathcal{A}_{10}$ & 9 & $1$, $2^3$, $4^5$\\ \hline
	$\mathcal{A}_{11}$ & 10 & $1$, $2^5$, $4^4$\\ \hline
	$\mathcal{A}_{12}$ & 10 & $1^3$, $3^6$, $6$\\ \hline
	$\mathcal{A}_{13}$ & 11 & $1^3$, $3^8$\\ \hline
	$\mathcal{A}_{14}$ & 14 & $1$, $2^{13}$\\ \hline
\end{tabular}
}

\caption{}
\end{table}

\begin{proof}
The statement of the lemma can be checked with the help of the GAP package COCO2P~\cite{GAP}.
\end{proof}

\begin{proof}[Proof of the Proposition~\ref{k1}]

From~\cite[Theorem~1, Lemma~5.5]{Ry1} it follows that the group $C_p^k$ is separable for $p\in\{2,3\}$ and $k\leq 2$. Let $\mathcal{A}$ be an $S$-ring over $G\cong C_p^3$, where $p\in\{2,3\}$. Then one of the statements of Lemma~\ref{descr} holds for $\mathcal{A}$. If Statement~1 of Lemma~\ref{descr} holds for $\mathcal{A}$ then, obviously, $\mathcal{A}$ is separable. If Statement~2 of Lemma~\ref{descr} holds for $\mathcal{A}$ then $\mathcal{A}$ is separable by Lemma~\ref{schurtens}. Suppose that Statement~3 of Lemma~\ref{descr} holds for $\mathcal{A}$. Then $\mathcal{A}$ is the proper schurian $S$-wreath product for some $\mathcal{A}$-section $S=U/L$ with $|S|\leq 3$. Since $\mathcal{A}$ is schurian, $\mathcal{A}_U$ is also schurian. Note that $\aut(\mathcal{A}_S)$ is 2-isolated becasue $|S|\leq 3$. Therefore $\mathcal{A}$ is separable by Lemma~\ref{2isol1}.

Suppose that Statement~4 of Lemma~\ref{descr} holds for $\mathcal{A}$ and $\varphi$ is an algebraic isomorphism from $\mathcal{A}$ to an $S$-ring $\mathcal{A}^{'}$ over an abelian group $G^{'}$. Clearly, if $\mathcal{A}^{'}$ is separable then  $\varphi^{-1}$ is induced by an isomorphism and hence $\varphi$  is also  induced by an isomorphism. If $G^{'}\cong C_{p^3}$  then $\mathcal{A}^{'}$ is separable by Theorem~\ref{main1}; if $G^{'}\cong C_p \times C_{p^2}$ then $\mathcal{A}^{'}$ is separable by~\cite[Theorem~1]{Ry1}; if $G^{'}\cong C_{p}^3$ and  one of the Statements~1-3 of Lemma~\ref{descr} holds for $\mathcal{A}^{'}$ then $\mathcal{A}^{'}$ is separable by the previous paragraph. So in the above cases $\varphi$ is induced by an isomorphism. Thus, we may assume that $G^{'}\cong C_p^3$ and Statement~4 of Lemma~\ref{descr} holds for $\mathcal{A}^{'}$. 

Two algebraically isomorphic $S$-rings have the same rank and sizes of basic sets. So information from Table~1 implies that $\mathcal{A}_i \ncong_{\alg} \mathcal{A}_j$ whenever $i\neq j$. Therefore we may assume that
$$\mathcal{A}=\mathcal{A}^{'}=\mathcal{A}_i$$
for some $i\in\{1,\ldots,14\}$. Using the package COCO2P again, one can find that 
$$|\aut_{\alg}(\mathcal{A}_j)|=|\iso(\mathcal{A}_j)|/|\aut(\mathcal{A}_j)|$$
for every $j\in\{1,\ldots,14\}$. In view of~(3) this yields that $\aut_{\alg}(\mathcal{A}_j)=\aut_{\alg}(\mathcal{A}_j)_0$ for every $j\in\{1,\ldots,14\}$. So $\varphi\in \aut_{\alg}(\mathcal{A}_i)_0$ and hence $\varphi$ is induced by an isomorphism. Thus, $\mathcal{A}$ is separable and the proposition is proved.
\end{proof}

\begin{prop}\label{k2}
The group $C_p\times C_p\times C_{p^k}$ is non-separable for $p\in\{2,3\}$ and $k\geq 2$.
\end{prop}

\begin{proof}

In view of Lemma~\ref{nonsepwr} to prove that the group $C_p\times C_p\times C_{p^k}$ is non-separable for $p\in\{2,3\}$ and $k\geq 2$ it is sufficient to construct an $S$-ring $\mathcal{A}$ over $C_p\times C_p\times C_{p^2}$, $p\in\{2,3\}$, and an algebraic isomorphism $\varphi$ from $\mathcal{A}$ to itself which is not induced by an isomorphism. 

Let $G=\langle a \rangle \times \langle b \rangle \times \langle c \rangle$, where $|a|=|b|=p$ and $|c|=p^2$. Put $A=\langle a \rangle$, $B=\langle b \rangle$, $C=\langle c \rangle$, $c_1=c^p$, and $C_1=\langle c_1 \rangle$. Firstly consider the case $p=2$. Let $f\in \aut(G)$ such that 
$$f:(a,b,c)\rightarrow (a, bac_1, ca)$$
and $\mathcal{A}=\cyc( \langle f \rangle, G)$. It easy to see that $|f|=2$ and the basic sets of $\mathcal{A}$ are the following
$$T_0=\{e\},~T_1=\{a\},~T_2=\{c_1\},~T_3=\{ac_1\},$$
$$X_1=cA,~X_2=c^3A,$$
$$Y_1=b\langle ac_1 \rangle,~Y_2=ba\langle ac_1 \rangle,$$
$$Z_1=bcC_1, Z_2=bcaC_1.$$
Define a permutation $\varphi$ on the set $\mathcal{S}(\mathcal{A})$ as follows:
$$T_0^{\varphi}=T_0,~T_1^{\varphi}=T_1,~T_2^{\varphi}=T_3,~T_3^{\varphi}=T_2,$$
$$X_1^{\varphi}=X_1,~X_2^{\varphi}=X_2,$$
$$Y_1^{\varphi}=Z_1, Y_2^{\varphi}=Z_2, Z_1^{\varphi}=Y_1,~Z_2^{\varphi}=Y_2.$$
It easy to see that $|\varphi|=2$. The straightforward check implies that $\varphi$ is an algebraic isomorphism from $\mathcal{A}$ to itself. Let us check, for example, that $c^{T_2}_{Y_1,Y_2}=c^{T_2^{\varphi}}_{Y_1^{\varphi},Y_2^{\varphi}}$. We have $\underline{Y_1}\underline{Y_2}=2a+2c_1$ and $\underline{Y_1}^{\varphi}\underline{Y_2}^{\varphi}=\underline{Z_1}\underline{Z_2}=2a+2ac_1$. So $c^{T_2}_{Y_1,Y_2}=c^{T_2^{\varphi}}_{Y_1^{\varphi},Y_2^{\varphi}}=2$. 

Note that $\mathcal{A}$ corresponds to a Kleinian quasi-thin scheme of index~4 in the sense of~\cite{MP2}. The $S$-ring $\mathcal{A}$ is cyclotomic and hence it is schurian. Assume that $\varphi$ is induced by an isomorphism. Then the algebraic fusion $\mathcal{A}^{\langle \varphi \rangle}$ is schurian by Lemma~\ref{fusion}. However, computer calculations made by using the package COCO2P~\cite{GAP} (see also~\cite{Ziv}) imply that $\mathcal{A}^{\langle \varphi \rangle}$ is non-schurian, a contradiction. Therefore, $\varphi$ is not induced by an isomorphism and hence $\mathcal{A}$ is non-separable.

Now let $p=3$. Let $f_1,f_2,f_3\in\aut(G)$ such that
$$f_1:(a,b,c) \rightarrow (a^{-1}, b^{-1}, c^{-1}),~f_2:(a,b,c) \rightarrow (a, b, cc_1),~f_3:(a,b,c) \rightarrow (a,ba,c).$$
The direct check implies that $|f_1|=2$, $|f_2|=|f_3|=3$, and $f_1,f_2,f_3$ pairwise commute. Put $K=\langle f_1 \rangle \times \langle f_2 \rangle \times \langle f_3 \rangle$ and $\mathcal{A}=\cyc(K,G)$. The basic sets of $\mathcal{A}$ are the following:
$$T_0=\{e\},~T_1=\{a,a^{-1}\},~T_2=\{c_1,c_1^{-1}\},~T_3=\{ac_1,a^{-1}c_1^{-1}\},~T_4=\{a^{-1}c_1,ac_1^{-1}\},$$
$$X_1=cC_1\cup c^{-1}C_1,~X_2=caC_1\cup c^{-1}a^{-1}C_1,~X_3=ca^{-1}C_1\cup c^{-1}aC_1,$$
$$Y_1=bA\cup b^{-1}A,~Y_2=bc_1A\cup b^{-1}c_1^{-1}A,~Y_3=b^{-1}c_1A\cup bc_1^{-1}A,$$
$$Z_1=\{bc,b^{-1}c^{-1}\}(A\times C_1),~Z_2=\{b^{-1}c,bc^{-1}\}(A\times C_1).$$
Let $\varphi$ be a permutation on the set $\mathcal{S}(\mathcal{A})$ such that $T_3^{\varphi}=T_4$, $T_4^{\varphi}=T_3$, and $X^{\varphi}=X$ for every $X\in\mathcal{S}(\mathcal{A})\setminus \{T_3,T_4\}$. Clearly, $|\varphi|=2$. Note that for every $X,Y\in\mathcal{S}(\mathcal{A})\setminus \{T_3,T_4\}$ the elements $\underline{T_3}$ and $\underline{T_4}$ enter with non-zero coefficients the element $\underline{X}\underline{Y}$ only in the following cases: $X=Y=Z_i$, $X=X_i,Y=X_j,~i\neq j$, $X=Y_i,Y=Y_j,~i\neq j$. The straightforward check using this observation implies that $\varphi$ is an algebraic isomorphism from $\mathcal{A}$ to itself. 

If $\varphi$ is induced by an isomorphism then $\mathcal{A}^{\langle \varphi \rangle}$ is schurian by Lemma~\ref{fusion}. However, $\mathcal{A}^{\langle \varphi \rangle}$ coincides with the non-schurian $S$-ring constructed in~\cite[pp. 8-10]{EKP} in case of $G\cong C_3\times C_3\times C_9$, a contradiction. Thus, $\varphi$ is not induced by an isomorphism and hence $\mathcal{A}$ is non-separable. The proposition is proved.

\end{proof}

Theorem~\ref{main2} is an immediate consequence of Proposition~\ref{k1} and Proposition~\ref{k2}.


\begin{thebibliography}{list}

\bibitem{E}
\emph{S.~Evdokimov}, Schurity and separability of association schemes, Thesis, 2004 (in Russian).


\bibitem{EP1}
\emph{S.~Evdokimov,~I.~Ponomarenko}, On a family  of Schur rings over a finite cyclic group, Algebra Analiz, \textbf{13}, No.~3 (2001), 139--154. 


\bibitem{EP2}
\emph{S.~Evdokimov,~I.~Ponomarenko}, Characterization of cyclotomic schemes and normal Schur rings over a cyclic group, Algebra Analiz, \textbf{14}, No.~2 (2002), 11-55.


\bibitem{EP3}
\emph{S.~Evdokimov,~I.~Ponomarenko}, Permutation group approach to association schemes, European J. Combin., \textbf{30}, No.~6 (2009), 1456--1476.



\bibitem{EP4}
\emph{S.~Evdokimov,~I.~Ponomarenko}, Schurity of $S$-rings over a cyclic group and generalized wreath
product of permutation groups, Algebra Analiz, \textbf{24}, No.~3 (2012), 84--127. 


\bibitem{EP5}
\emph{S.~Evdokimov,~I.~Ponomarenko}, Coset closure of a circulant S-ring and schurity problem,  J. Algebra Appl., \textbf{15}, No.~4 (2016), 1650068-1-1650068-49.


\bibitem{EP6}
\emph{S.~Evdokimov,~I.~Ponomarenko}, On separability problem for circulant $S$-rings, Algebra Analiz, \textbf{28}, No.~1 (2016), 32--51.



\bibitem{EKP}
\emph{S.~Evdokimov,~I.~Kov\'acs,~I.~Ponomarenko}, On schurity of finite abelian groups, Commun. Algebra,  \textbf{44}, No.~1 (2016), 101--117.


\bibitem{KPS} 
\emph{S.~Kiefer, I.~Ponomarenko, P.~Schweitzer},  The Weisfeiler-Leman dimension of planar graphs is at most~3, 32nd Annual ACM/IEEE Symposium on Logic in Computer Science (LICS), 12 pp., IEEE, [Piscataway], NJ, 2017.





\bibitem{GAP}
\emph{M.~Klin, C.~Pech, S.~Reichard}, COCO2P -- a GAP package, 0.14, 07.02.2015, http://www.math.tu-dresden.de/~pech/COCO2P.


\bibitem{Klin}
\emph{J.~Golfand, M.~Klin}, Amorphic cellular rings I, in: Investigations in Algebraic Theory of Combinatorial Objects, VNIISI, Institute for System Studies, Moscow, 1985, 32-38.



\bibitem{MP}
\emph{M.~Muzychuk,~I.~Ponomarenko}, Schur rings, European J. Combin., \textbf{30}, No.~6 (2009), 1526--1539.

\bibitem{MP2}
\emph{M.~Muzychuk,~I.~Ponomarenko}, On quasi-thin association schemes, J. Algebra, \textbf{351}, No.~1 (2012), 467-489.


\bibitem{MP3}
\emph{M.~Muzychuk,~I.~Ponomarenko},  On Schur $2$-groups, Zapiski Nauchnykh Seminarov POMI, \textbf{435} (2015), 113-162. 


\bibitem{Po}
\emph{R.~P\"{o}schel}, Untersuchungen von S-Ringen insbesondere im Gruppenring von p-Gruppen, Math. Nachr., 60 (1974), 1-27.


\bibitem{Ry1} 
\emph{G.~Ryabov},  On separability of Schur rings over abelian $p$-groups, Algebra and logic, \textbf{57}, No.~1 (2018), 73--101.


\bibitem{Ry2} 
\emph{G.~Ryabov},  Separability of Schur rings over an abelian group of order~$4p$,  Zapiski Nauchnykh Seminarov POMI , \textbf{470} (2018), 179--193.



\bibitem{Schur}
\emph{I.~Schur}, Zur theorie der einfach transitiven Permutationgruppen, S.-B. Preus Akad. Wiss.
Phys.-Math. Kl., \textbf{18}, No.~20 (1933), 598--623.


\bibitem{Wi}
\emph{H.~Wielandt}, Finite permutation groups, Academic Press, New York - London, 1964.



\bibitem{WeisL}
\emph{B.~Weisfeiler,~A.~Leman}, Reduction of a graph to a canonical form and an algebra which appears in the process, NTI, \textbf{2}, No.~9 (1968), 12--16.


\bibitem{Ziv}
\emph{M.~Ziv-Av},  Enumeration of Schur rings over small groups, CASC Workshop 2014, LNCS 8660 (2014), 491-500.


\end{thebibliography}
\end{document}